\documentclass[a4paper,11pt]{article}
\usepackage[T1]{fontenc}
\usepackage[utf8]{inputenc}
\usepackage[english]{babel}
\usepackage[stretch=10]{microtype}
\usepackage{mathtools}%carica automaticamente amsmath
\usepackage{amssymb}%carica automaticamente amsfonts
\usepackage{bm}
\usepackage{enumitem}
\usepackage{xcolor}
\usepackage{comment}
\usepackage[autostyle]{csquotes}
\usepackage[backend=bibtex,doi=false,isbn=false,url=false,maxbibnames=9]{biblatex}
\usepackage{hyperref}%importante l'ordine hyperref --> amsthm&geometry --> cleveref
\hypersetup{
	colorlinks,
	linkcolor={red!50!black},
	citecolor={blue!50!black},
	urlcolor={blue!80!black}
}
\usepackage{amsthm}
\usepackage{geometry}
\usepackage{cleveref}%ultimo
\DeclarePairedDelimiter\abs{\lvert}{\rvert}
\DeclarePairedDelimiter\norm{\lVert}{\rVert}
\renewcommand*{\P}{\mathbf{P}}
\newcommand*{\E}{\mathbf{E}}
\newcommand*{\numberset}{\mathbb}
\newcommand*{\N}{\numberset{N}}

\newcommand*{\R}{\numberset{R}}
\newcommand*\de{\mathop{}\!\mathrm{d}}
\newcommand*{\indic}{\bm{1}}
\newcommand{\oneover}[1]{\frac{1}{#1}}
\renewcommand{\epsilon}{\varepsilon}
\renewcommand{\phi}{\varphi}
\theoremstyle{plain}
\newtheorem{theorem}{Theorem}[section]

\newtheorem{lemma}[theorem]{Lemma}
\newtheorem{proposition}[theorem]{Proposition}
\theoremstyle{definition}

\theoremstyle{remark}
\newtheorem{remark}{Remark}
\newcommand*{\B}{\mathbb{B}}
\newcommand*{\vol}{\abs}
\renewcommand{\S}{\mathbb{S}^{n-1}}
\DeclarePairedDelimiterXPP\psinorm[2]{}{\lVert}{\rVert}{_{L^{\psi_2}(\mu_{#1})}}{#2}
\DeclarePairedDelimiterXPP\maxnorm[1]{}{\lVert}{\rVert}{_\infty}{#1}
\newcommand{\xtheta}{\langle x,\theta\rangle}
\newcommand{\dottheta}{\langle \cdot,\theta\rangle}

\allowdisplaybreaks[4]%allows to split multi-line equations between pages
\title{The isotropic constant of random polytopes\\ with vertices on convex surfaces}
\author{Joscha Prochno, Christoph Th\"ale and Nicola Turchi}
\date{}
\begin{document}
	\maketitle
	
\begin{abstract}
For an isotropic convex body $K\subset\mathbb{R}^n$ we consider the isotropic constant $L_{K_N}$ of the symmetric random polytope $K_N$ generated by $N$ independent random points which are distributed according to the cone probability measure on the boundary of $K$. We show that with overwhelming probability $L_{K_N}\leq C\sqrt{\log(2N/n)}$, where $C\in(0,\infty)$ is an absolute constant. If $K$ is unconditional we argue that even $L_{K_N}\leq C$ with overwhelming probability. The proofs are based on concentration inequalities for sums of sub-exponential or sub-Gaussian random variables, respectively, and, in the unconditional case, on a new $\psi_2$-estimate for linear functionals with respect to the cone measure in the spirit of Bobkov and Nazarov, which might be of independent interest. 

\bigskip


\end{abstract}	

%\tableofcontents
	
\section{Introduction and main results}

\subsection{Motivation, background and overview}

The study of random polytopes began with the work of Sylvester and the famous four-point problem more than 150 years ago \cite{Sylvester}. This problem asks for the probability that four randomly chosen points in a (possibly infinite) planar region have a convex hull which is a quadrilateral. Its solution depends on the probability distribution of the random points and was the starting point for an extensive study. Much later, in their groundbreaking work \cite{RS1963}, R\'enyi and Sulanke continued studying expectations of various basic functionals of random polytopes in the plane. From a methodical point of view, the study of random polytopes combines ideas and techniques from several areas of mathematics such as convex and discrete geometry, geometric functional analysis, or probability theory (see \cite{BaranySurvey,HugSurvey,Reitzner2010} for surveys).

In the last 30 years a tremendous effort has been made to explore various properties of random polytopes as they gained more and more importance due to numerous applications and connections to various other fields. These can be found not only in statistics (in form of extreme points of random samples), computational geometry (when approximating convex sets) or numerical analysis (in the context of numerical integration), but also in computer science in the analysis of the average complexity of algorithms \cite{PS1990} or in optimization \cite{B1987}, when simplex algorithms are considered. A particularly celebrated result due to Spielman and Teng \cite{ST2004} is the smoothed analysis of algorithms, in which one measures the expected complexity of an algorithm under slight random perturbations of arbitrary inputs. They showed that the shadow vertex simplex method has polynomial smoothed complexity and their bounds were later improved by Vershynin in \cite{V2009}. Again, random polytopes play an important r\^ole. Another application of (high-dimensional) polytopes appears when studying graphs with many vertices. In the classical (i.e., deterministic setting), this is related to what is called polyhedral combinatorics, where one usually considers, for instance, the convex hull of the characteristic vectors of the matchings in the graph $G$. The corresponding polytope is then called the matching polytope of $G$. The idea of polyhedral combinatorics (in the case of matchings) would then be to find inequalities describing the matching polytope's facets. This may provide insights into the combinatorial structure of the matchings with often algorithmic consequences. For instance, knowing the facets we can optimize any linear function over the polytope in polynomial time given the facets have a `nice' structure. For more details, we highly recommend the book of Matou\v sek \cite{M2002} and also refer to the references cited therein. 

In 1989, Milman and Pajor \cite{MP} revealed a deep connection between random polytops and geometric functional analysis by proving that the expected volume of a certain random simplex is closely related to the isotropic constant of a convex set. In fact, this is a fundamental quantity in convex geometry and the local theory of Banach spaces \cite{MP}. Very recently, Hinrichs, Prochno and Ullrich discovered an interesting and promising connection of the isotropic constant of the domain of integration and the tractability analysis of high-dimensional numerical integration on that domain for certain classes of smooth functions \cite{HPU2018}.  

Today, the geometry of convex bodies in high dimensions is an active area of current mathematical research and has given birth to a new area, called asymptotic geometric analysis. It lies at the crossroad between convex and discrete geometry, functional analysis, and probability theory. In particular, it has turned out that the presence of high dimensions forces a certain regularity behavior for the involved convex bodies. The goal of this paper is to explore such regularity phenomena further and study volumetric properties (which are related to the isotropic position and the isotropic constant) of a large class of random convex sets.

\medbreak

Let us recall that a convex body $K\subset\R^n$ of unit volume is isotropic if its barycenter is at the origin and its inertia matrix is a constant multiple $L_K^2$ of the identity matrix (all notions will formally be introduced in the next subsection or in \Cref{sec:Preliminaries} below). The constant $L_K$ is the isotropic constant of the body $K$ and the question is whether or not there exists an absolute constant $C\in(0,\infty)$ such that $L_K\leq C$ for all space dimensions $n\in\N$ and all isotropic convex bodies $K\subset\R^n$. While this problem is still open in its general form (the currently best bound is $L_K\leq C\sqrt[4]{n}$, due to Klartag \cite{K06}), the isotropic constant of several special classes of convex bodies is in fact known to be bounded. Examples include zonoids and duals of zonoids \cite{B}, unconditional convex bodies \cite{Bou,KMP} and unit balls of Schatten classes \cite{KMP}. Against this background, Klartag and Kozma \cite{KK} started to investigate the isotropic constant of random convex sets, as it is known since the groundbreaking work of Gluskin on the Banach-Mazur compactum \cite{G} that random constructions often display some kind of extremal behaviour. Their ideas were taken up by Alonso-Guti\'errez \cite{A}, Alonso-Guti\'errez, Litvak and Tomczak-Jaegermann \cite{ALT}, Dafnis, Giannopoulos and Gu\'edon \cite{DGG} and H\"orrmann, Prochno and Th\"ale \cite{HPT17} to prove boundedness of the isotropic constant for several classes of random polytopes with probability tending to $1$, as the space dimension tends to infinity. An entirely different approach was used by H\"orrmann, Hug, Reitzner and Th\"ale \cite{HHRT} for zero cells of a class of Poisson hyperplane tessellations.

The present paper can be regarded as a natural continuation of \cite{A} and \cite{HPT17}, where random polytopes generated by random points on \(\ell_p\text{-spheres}\) have been investigated. Here we take a more general point of view and consider random convex hulls whose points are distributed according to the cone (probability) measure on a convex surface, i.e., on the boundary of an arbitrary (isotropic) convex body $K\subset\R^n$. Against this background our paper can also be regarded as a complement to \cite{ALT,DGG}, where the random points were selected uniformly at random from the interior of $K$. More precisely, we shall prove 
\begin{itemize}
\item[(i)] that the isotropic constant $L_{K_N}$ of a random polytope generated by $n<N<e^{\sqrt{n}}$ independent random points on the boundary of an isotropic convex body $K\subset\R^n$ satisfies
	\begin{equation*}
		L_{K_N} \leq C\sqrt{\log\frac{2N}{n}}
	\end{equation*}
with probability at least $1-c_1e^{-c_2n}-e^{-c_3\sqrt{N}}$ for absolute constants $C,c_1,c_2,c_3\in(0,\infty)$;
\item[(ii)] that if $K$ is in addition symmetric with respect to all coordinate hyperplanes (i.e., if $K$ is unconditional), we even have that
\begin{equation*}
L_{K_N} \leq C
\end{equation*}
with probability bounded below by $1-c_1e^{-c_2n}$ for all $N>n$.
\end{itemize}

The result (i) for general $K$ resembles the so-far best known upper bound for the isotropic constant of random convex hulls in \cite{ALT}, where the generating points were selected with respect to the uniform distribution on $K$. Similarly, our result (ii) is the analogue to the main finding in \cite{DGG}, where boundedness of the isotropic constant of random convex hulls was obtained in the unconditional case. However, we emphasize that as in \cite{A,HPT17} our bounds cannot be concluded from those in the existing literature, since the cone probability measure on the boundary of an isotropic convex body is not log-concave. To study the isotropic constant of random polytopes for which the generating measure is not log-concave was in fact the main source of motivation for this work and its predecessors \cite{A,HPT17}.

\medbreak

The remaining part of the paper is structured as follows. In \Cref{sec:Unconditional} we present our main result for isotropic unconditional convex bodies, while our estimate for the general case is the content of \Cref{sec:General}. Some preliminaries and auxiliary results are collected in \Cref{sec:Preliminaries} and the proofs are presented in \Cref{sec:ProofIsotropicConstant} and \Cref{sec:ProofGeneral}. A new Bobkov-Nazarov \(\psi_2\text{-type}\) estimate for the cone measure, which is the key point in our derivation of the result for isotropic unconditional convex bodies, is presented in \Cref{sec:Psi2}.

\subsection{Main results I -- the unconditional case}
\label{sec:Unconditional}

We work in the $n$-dimensional Euclidean space $\R^n$ and call a compact and convex subset of $\R^n$ with non-empty interior a convex body. A convex body $K$ is said to be isotropic provided that $\vol{K}=1$, its barycenter is at the origin and
	\begin{equation*}
		\int_K\xtheta^2\de x = L_K^2
	\end{equation*}
for all directions $\theta\in\mathbb{S}^{n-1}$, where $L_K$ is a constant independent of $\theta$, the so-called isotropic constant of $K$. Here and in what follows $\vol{\,\cdot\,}$ denotes the ($n$-dimensional) Lebesgue measure and $\langle\,\cdot\,,\cdot\,\rangle$ is the standard scalar product. In addition, we call $K$ unconditional provided that $K$ is symmetric with respect to all $n$ coordinate hyperplanes. In particular, an unconditional convex body is symmetric. If $K\subset\R^n$ is a convex body with boundary $\partial K$ we let $\mu_K$ be the cone probability measure on $\partial K$, i.e.,
	\begin{equation*}
		\mu_K(B)\coloneqq\frac{\vol{\{rx:x\in B,0\leq r\leq 1\}}}{\vol{K}},\qquad B\subset\partial K\text{ a Borel set}.
	\end{equation*}
For $N> n$ we let $X_1,\ldots,X_N$ be independent random points distributed according to $\mu_K$ and $K_N\coloneqq\mathrm{conv}\{\pm X_1,\ldots,\pm X_N\}$ be the random symmetric convex hull generated by $X_1,\ldots,X_N$. Our first main result is the following one.

\begin{theorem}\label{thm:IsotropicConstant}
Let $K\subset\R^n$ be an isotropic unconditional convex body, $N>n$ and $K_N$ the symmetric convex hull of $N$ independent random points on \(\partial K\) with distribution $\mu_K$. Then there exist absolute constants $c_1,c_2,C\in(0,\infty)$ such that the event that
\begin{equation*}
L_{K_N}\leq C
\end{equation*}
occurs with probability at least \(1-c_1e^{-c_2n}\). 
\end{theorem}

\begin{remark}\label{rem:N<cn}
We note that the result of \Cref{thm:IsotropicConstant}, in the regime where $N$ is proportional to the space dimension $n$, follows directly from the existing literature. Indeed, in this situation the random polytope $K_N$ has precisely $N$ vertices with probability one and it is known from \cite{ABBW} that an $n$-dimensional polytope $P$ with $v_P> n$ vertices has an isotropic constant bounded above by a constant multiple of $\sqrt{v_P/n}$. This implies absolute boundedness of the isotropic constant of $K_N$ even with probability one.
\end{remark}

Let us emphasize that Theorem \ref{thm:IsotropicConstant} generalizes the main results of both \cite{A} and \cite{HPT17}. Moreover, it is the clear analogue to the main result in \cite{DGG}, where the authors consider random polytopes generated by points $X_1,\ldots,X_N$ chosen uniformly at random from the interior of an isotropic unconditional convex body. However, the result in \cite{DGG} does not imply Theorem \ref{thm:IsotropicConstant} and vice versa. Although the tools we use and the strategy of the proof rely on similar ingredients as those employed in \cite{DGG} (and also that in \cite{A,ALT,KK}) there is a significant difference. In fact, one of the main ingredients in the proof of \Cref{thm:IsotropicConstant} is a version of Bernstein's inequality (see \Cref{lem:bernstein} below). In order to be able to apply it, an upper bound on the so-called $\psi_2$-norm of linear functionals with respect to the cone probability measure is needed. While this is well known in the case of the uniform distribution on $K$, this is not the case for the cone probability measure on $\partial K$, the main reason for this being the fact that the cone measure does not fit into the theory developed for log-concave measures. We shall provide such an estimate in \Cref{sec:Psi2}.

\subsection{Main results II -- the general case}\label{sec:General}

We assume the same set-up as in the previous subsection, but now we drop the assumption that the convex body $K$ is unconditional. That is, we assume that $K\subset\R^n$ is an isotropic convex body with cone probability measure $\mu_K$. The next result is the analogue to the main result in \cite{ALT}, where the authors prove a similar estimate in the case that the random polytope is generated by points uniformly distributed in the interior of $K$. However and in contrast to \Cref{thm:IsotropicConstant}, for general isotropic convex bodies we are (in general) not able to bound the isotropic constant of $K_N$ by an absolute constant with high probability. In addition this set-up requires an upper bound for the number of vertices of $K_N$. Again, this is in line with the results in \cite{ALT}.

\begin{theorem}\label{thm:IsotropicConstantGeneralK}
Let $K\subset\R^n$ be an isotropic convex body, $n<N\leq e^{\sqrt{n}}$ and $K_N$ the symmetric convex hull of $N$ independent random points on \(\partial K\) with distribution $\mu_K$. Then there exist absolute constants $c_1,c_2,c_3,C\in(0,\infty)$ such that the event that 
\begin{equation*}
L_{K_N}\leq C\sqrt{\log\frac{2N}{n}}
\end{equation*}
occurs with probability at least \(1-c_1e^{-c_2n}-e^{-c_3\sqrt{N}}\).
\end{theorem}

\begin{remark}
As in \Cref{rem:N<cn}, if $N\leq cn$ for some $c\in(0,\infty)$ the conclusion of \Cref{thm:IsotropicConstantGeneralK} is again trivial. More precisely, in this regime we even have that $L_{K_N}$ is absolutely bounded with probability one.
\end{remark}

\begin{remark}
Let us point out that in the regime where $e^{\sqrt{n}}<N\leq e^{n}$ one can prove that the weaker estimate
\begin{equation*}
L_{K_N}\leq CL_K\sqrt{\log\frac{2N}{n}}
\end{equation*}
holds with probability exponentially close to $1$. This follows from the fact that one can use part (a) of \Cref{lem:volrad} instead of (b) to lower bound $\vol{K_N}^{1/n}$ in the final proof (see also the discussion after Theorem 11.3.7 in \cite{IsotropicConvexBodies}).
\end{remark}

\begin{remark}
It was shown in \cite{ABBW} that if $P$ is a polytope in $\R^n$ with $f_P$ facets then
\begin{align}\label{eq:IsoGeneralPolytopeFGiven}
L_P\leq C\sqrt{\log\frac{f_P}{n}}
\end{align}
for some absolute constant $C\in(0,\infty)$. Moreover, in \cite{Stemeseder} (see also \cite{BuchtaMuellerTichy} for the case of the unit ball) it is proved that if $\partial K$ is twice differentiable and has positive Gaussian curvature everywhere the expected number of facets of $K_N$ satisfies
\begin{equation*}
\E f_{K_N} = cN(1+o(1)),
\end{equation*}
as $N\to\infty$, where $c\in(0,\infty)$ is some constant depending on $d$ and on $K$, and $o(1)$ is some sequence that tends to zero. (The results in \cite{BuchtaMuellerTichy,Stemeseder} are formulated for the non-symmetric convex hull of $N$ random points in $K$, but it can be checked that the order remains the same for the symmetric convex hulls.) Thus, replacing $f_P$ by $\E f_{K_N}$ in \eqref{eq:IsoGeneralPolytopeFGiven}, the result of \Cref{thm:IsotropicConstantGeneralK} might be anticipated.
\end{remark}

The proof of \Cref{thm:IsotropicConstantGeneralK} is similar to the one of \Cref{thm:IsotropicConstant}, but is based on another version of Bernstein's inequality. More precisely, while in the argument for \Cref{thm:IsotropicConstant} we work with the so-called $\psi_2$-norm (of a certain class of linear functionals), in the context of \Cref{thm:IsotropicConstantGeneralK} we are able to deal only with the $\psi_1$-norm, which can effectively be handled for arbitrary isotropic convex bodies.

\section{Preliminaries}\label{sec:Preliminaries}

In this section we gather some preliminary material which is needed in our arguments below.

	\subsection{Basic notation}
	
	Let \(\N=\{1,2,\ldots\}\) be the set of natural numbers. For \(n\in\N\), we work in the $n$-dimensional Euclidean space $\R^n$ with standard inner product $\langle\,\cdot\,,\cdot\,\rangle$ and induced norm $\norm{\,\cdot\,}_2$. More generally, for $p\in[1,\infty]$ we introduce the $p$-norm of $x=(x_1,\ldots,x_n)\in\R^n$ by putting
	\begin{equation*}
	\norm{x}_p \coloneqq\begin{cases}
	\Bigl(\sum\limits_{i=1}^n\abs{x_i}^p\Bigr)^{1/p} &:\, p<\infty\\
	\max\{\abs{x_1},\ldots,\abs{x_n}\} &:\, p=\infty.
	\end{cases}
	\end{equation*}
	By $\B_p^n$ we denote the unit ball in $\R^n$ with respect to the $p$-norm and we let $\mathbb{S}_p^{n-1}$ be the corresponding $p\text{-sphere}$. For the special case $p=2$, we shall write $\mathbb{S}^{n-1}$ instead of $\mathbb{S}_2^{n-1}$. 
	%and $\sigma$ for the spherical Lebesgue measure on $\mathbb{S}^{n-1}$. We remark that $\sigma$ coincides with the cone probability measure $\mu_{\B_2^n}$ on $\B_2^n$.
	% More generally, 
	It is known that $\mu_{\B_p^n}$ coincides with the surface measure on $\mathbb{S}_p^{n-1}$ (i.e., the $(n-1)$-dimensional Hausdorff measure $\mathcal{H}^{n-1}$ on $\mathbb{S}_p^{n-1}$) if and only if $p\in\{1,2,\infty\}$. The total variation distance between these two measures was studied in \cite{N}.
	
	Given sets \(A\subset\R^n\) and \(I\subset[0,\infty)\), we define the set \(IA\subset\R^n\) as 
		\begin{equation*}
			IA\coloneqq\big\{rx\in\R^n:r\in I, x\in A \big\}. 
		\end{equation*}
 	When \(I=\{r\}, r\in[0,\infty) \) we also write \(rA\) instead of \(\{r\}A \). Moreover, \(\mathrm{conv}\,A\) will denote the convex hull of $A$.
 	
	 There exists a norm associated to any symmetric convex body \(K\), called the Minkowski functional of \(K\). It is defined for every \(x\in\R^n \) as
 	\begin{equation*}
 		\norm{x}_K\coloneqq\inf\{r>0: x\in rK \}.
 	\end{equation*}
 	Note, in particular that \(\norm{x}_K=1\) if and only if \(x\in\partial K\).
 	%, i.e., if $x$ is located on the boundary $\partial K$ of $K$.
 	
 	 Within the present paper we use the convention that $C,c,c_1,c_2,\ldots$ denote absolute constants whose value might change across occurrences.
 	
	\subsection{Orlicz spaces and Bernstein's Inequality}
	
	Fix a probability space \((\Omega,\mathcal{F},\P)\). A convex function \(M\colon[0,\infty)\to[0,\infty) \) with \(M(0)=0\) is called an Orlicz function. We indicate by \(L^M(\Omega,\P)\) the set (of equivalence classes) of random variables \(X:\Omega\to\R\) such that \(M(\abs{X}/\lambda)\in L^1(\Omega,\P)\), for some \(\lambda>0 \). Here, $L^1(\Omega,\P)$ stands for the set of (equivalence classes) of integrable random variables on $\Omega$. We supply \(L^M(\Omega,\P)\) with the Luxemburg norm
	\begin{equation*}
		\norm{X}_M\coloneqq\inf\{\lambda>0:\E\,M(\abs{X}/\lambda)\le 1\},
	\end{equation*}
	where $\E$ stands for the expectation (i.e., integration) with respect to $\P$ (this notation should not lead to confusion with the Minkowski $\norm{\,\cdot\,}_K$ functional associated with a convex body $K$). Let us point out that \((L^M(\Omega,\P),\norm{\,\cdot\,}_M)\) is a Banach space and we refer to it as the Orlicz space associated to \(M\). Examples of Orlicz spaces are the \(L^p\text{-spaces}\), for every \(p\in[1,
	\infty)\), associated to the Orlicz functions \(x\mapsto x^p\), and the spaces associated to the functions \(\psi_\alpha(x)=\exp(x^\alpha)-1 \), for every \(\alpha\in[1,
	\infty)\). In the particular case \(\alpha\in\{1,2\} \), the elements of the space \((L^{\psi_\alpha}(\Omega,\P),\norm{\,\cdot\,}_{\psi_\alpha})\) are also called sub-exponential and sub-Gaussian random variables, respectively.
	
	The following result, known as Bernstein's inequality, taken in this form from \cite[Theorem 3.5.17]{AsymGeomAnalyBook}, allows to obtain an estimate on the tail of the distribution of a sum of independent and uniformly sub-Gaussian random variables. It will be used in the proof of \Cref{thm:IsotropicConstant}. In the proof of \Cref{thm:IsotropicConstantGeneralK}, we need another version of Bernstein's inequality, which deals with sums of independent and uniformly sub-exponential random variables. It is written here as a particular case of \cite[Theorem 3.5.16]{AsymGeomAnalyBook}.
	
	\begin{lemma}
		\label{lem:bernstein}
		Let \(Y_1,\ldots,Y_n\) be independent and centred random variables defined on a common probability space \((\Omega,\mathcal{F},\P)\).
		\begin{itemize}
		\item[(a)] Suppose that there exists \(R\in(0,\infty)\) such that \(\norm{Y_i}_{\psi_2}\le R\) for every \(i\in\{1,\ldots, n\} \). Then, for every \(t>0\),
				\begin{equation*}
				\P \Bigl(\abs[\Big]{\sum_{i=1}^n Y_i}>t n\Bigr)\le 2\exp\Bigl(-\frac{t^2 n}{8 R^2} \Bigr).
				\end{equation*}
		\item[(b)] Suppose that there exists \(R\in(0,\infty)\) such that \(\norm{Y_i}_{\psi_1}\le R\) for every \(i\in\{1,\ldots, n\} \). Then, for every \(t>0\),
				\begin{equation*}
				\P \Bigl(\abs[\Big]{\sum_{i=1}^n Y_i}>t n\Bigr)\le 2\exp\Bigl(-\frac{tn}{6R}\min\Bigl\{\frac{t }{R},1\Bigr\} \Bigr).
				\end{equation*}
		\end{itemize}
	\end{lemma}

	\subsection{Geometry of convex bodies}
	%We call \(K\subset\Rn\) a convex body when it is a compact convex set with non-empty interior. We indicate the Lebesgue measure of \(K\) as \(\vol{K} \), which is also called the volume of \(K\). A convex body is symmetric if \(x\in K\) implies \(-x\in K\). 
	
	After having defined the isotropic constant of an isotropic convex body in the introduction, let us recall that the isotropic constant \(L_K\) can be defined also for an arbitrary convex body \(K\subset\R^n\) as follows:
	\begin{equation}
	\label{eq:defLK}
	L_K^2\coloneqq\min\biggl\{\oneover{n\vol{TK}^{1+2/n}}\int_{z+TK}\norm{x}_2^2\de x : z\in\R^n, T\in\mathrm{GL}(n) \biggr\},
	\end{equation}
	where $\mathrm{GL}(n)$ stands for the group of invertible linear transformations on $\R^n$, see \cite[Definition 10.1.6]{AsymGeomAnalyBook}. Although this definition relies on the $2$-norm, the isotropic constant of a symmetric convex body can be bounded from above using an average of the $1$-norm. As in \cite{DGG} this bound will turn out to be very useful for our purposes. The first of the following inequalities is taken from \cite[Lemma 11.5.2]{IsotropicConvexBodies}, while the second is a direct consequence of the definition of isotropic constant (\Cref{eq:defLK}). We recall that a convex body $K\subset\R^n$ is symmetric provided that $x\in K$ implies $-x\in K$ and centred if $K$ has its barycentre at the origin.
	
	\begin{lemma}\label{lem:boundLK}
	\begin{itemize}
	\item[(a)] Let \(K\subset\R^n\) be a symmetric convex body. Then there exists a  constant \(c\in(0,\infty)\) such that
			\begin{equation*}
			L_K\le\frac{c}{n\vol{K}^{1+1/n}}\int_K\norm{x}_1\de x.
			\end{equation*}
	\item[(b)] Let \(K\subset\R^n\) be a centred convex body. Then,
				\begin{equation*}
						L_K^2\le\oneover{n\vol{K}^{1+2/n}}\int_K\norm{x}_2^2\,\de x.
				\end{equation*}
	\end{itemize}
	\end{lemma}

Since we shall be dealing mostly with symmetric polytopes, we will make use of the following lemma that, together with the previous one, allows us to connect the isotropic constant of a polytope with properties of its facets. Part (a) of the next result follows directly from \cite[Lemma 11.5.4]{IsotropicConvexBodies} and \cite[Identity (2.26)]{DGG}, while part (b) is a consequence of \cite[Lemma 11.4.4]{IsotropicConvexBodies} and \cite[Lemma 11.4.5]{IsotropicConvexBodies}.
	
	\begin{lemma}\label{lem:facets}
	\begin{itemize}
	\item[(a)] 		Let \(K\subset\R^n\) be a symmetric polytope. Then
			\begin{equation*}
			\oneover{\abs{K}}\int_{K}\norm{x}_1\de x\le\frac{1+\sqrt{2}}{n}\max_{\substack{{\mathrm{conv}}\{y_1,\ldots,y_n\}\text{\emph{ is a  facet of }} K\\ \epsilon_1,\ldots,\epsilon_n=\pm 1}}\norm{\epsilon_{1}y_{1}+\ldots+\epsilon_{n}y_{n}}_1.
			\end{equation*}
	\item[(b)] 			Let \(K\subset\R^n\) be a symmetric polytope. Then
			\begin{equation*}
			\oneover{\abs{K}}\int_{K}\norm{x}_2^2\,\de x\le\frac{2}{(n+1)(n+2)}\max_{\substack{{\mathrm{conv}}\{y_1,\ldots,y_n\}\text{\emph{ is a  facet of }} K\\ \epsilon_1,\ldots,\epsilon_n=\pm 1}}\norm{\epsilon_{1}y_{1}+\ldots+\epsilon_{n}y_{n}}_2^2.
			\end{equation*}
	\end{itemize}	
	\end{lemma}

	\section{A $\psi_2$-estimate for the cone measure}\label{sec:Psi2}
	
	In order to be able to apply Bernstein's inequality for independent and uniformly sub-Gaussian random variables (see \Cref{lem:bernstein} (a)), we need an upper bound on the \(\psi_2\)-norm on linear functionals with respect to the cone probability measure on the boundary of an isotropic unconditional convex body. We emphasize that such an estimate is the key point in the proof of \Cref{thm:IsotropicConstant} and might also be of independent interest. Bounds for the $\psi_2$-norm of linear functionals have been subject of a number of studies, which in turn concentrate on the case of the uniform distribution on an isotropic convex body or, more generally, on an isotropic log-concave measure, see in particular the work of Bobkov and Nazarov \cite{BobkovNazarov}. However, the cone measure does clearly not satisfy this property and it seems that the following Bobkov-Nazarov type estimate for the cone measure result is not available in the existing literature. We also remark that the proof in \cite{BobkovNazarov} does not carry over to the cone measure case. Instead we use the polar integration formula to deduce the estimate from the one for the uniform distribution. In addition, this allows to identify the extremal bodies for which the estimate is sharp, see \Cref{rem:ExtremalBodies}.
	
	\begin{proposition}
		\label{prop:psi2}
		Let \(K\subset\R^n\) be an isotropic unconditional convex body. Then, for every \(\theta\in\R^n\),
		\begin{equation*}
		\psinorm{K}{\dottheta}\le 3\sqrt{n}\maxnorm{\theta}.
		\end{equation*}
	\end{proposition}
	
	Let us briefly comment that the result of \Cref{prop:psi2} might be re-phrased by saying that for an isotropic unconditional convex body $K$ and for every $\theta\in\R^n$ one has that
	\begin{equation*}
	\mu_K(\{x\in\partial K:\abs{\xtheta}\geq t\sqrt{n}\norm{\theta}_\infty\}) \leq 2e^{-\frac{t^2}{9}}
	\end{equation*}
	for all $t>0$. Especially, taking $\theta=(1,\ldots,1)$, which satisfies $\norm{\theta}_\infty=1$, we have that
	\begin{equation*}
	\mu_K\Bigl(\Bigl\{x\in\partial K:\frac{\norm{x}_1}{\sqrt{n}}\geq t\Bigr\}\Bigr) \leq 2e^{-\frac{t^2}{9}}.
	\end{equation*}
	We split the proof of \Cref{prop:psi2} into three lemmas, the first being a comparison inequality, in the spirit of \cite{BobkovNazarov}, where we compare the absolute moments of a linear functional on a general isotropic unconditional convex body to the ones on a rescaling of the unit ball $\B_1^n$.
	
	\begin{lemma}
		\label{lem:comparison}
		Let \(K\subset\R^n\) be an isotropic unconditional convex body and \(V\coloneqq\frac{\sqrt{6}}{2}n\B_1^n\). Then, for any \(\theta\in\R^n\) and every \(q\in\N\cup\{0\}\),
		\begin{equation*}
		\int_{\partial K}\abs{\langle x,\theta\rangle}^{2q}\de\mu_K(x)\le 	\int_{\partial V}\abs{\langle x,\theta\rangle}^{2q}\de\mu_V(x).
		\end{equation*}
	\end{lemma}
	\begin{proof}
		We know from the computations in \cite{BobkovNazarov} (see in particular \cite[page 307]{IsotropicConvexBodies}) that, for any \(\theta\in\R^n\),
		\begin{equation}
		\label{eq:1}
		\int_{\R^n}\abs{\xtheta}^{2q}\de\nu_K(x)\le 	\int_{\R^n}\abs{\xtheta}^{2q}\de\nu_V(x),
		\end{equation}
		since $K$ is unconditional. Then the claim holds if for any symmetric convex body \(K_0\),
		\begin{equation}
		\label{eq:2}
		\int_{\R^n}\abs{\xtheta}^{2q}\de\nu_{K_0}(x)=c_{n,q} 	\int_{\partial K_0}\abs{\langle x,\theta\rangle}^{2q}\de\mu_{K_0}(x),
		\end{equation}
		where \(c_{n,q}\in(0,\infty)\) can depend on \(n,q\) but not \(K_0\).
		We can prove \Cref{eq:2} using a polar integration formula for the cone measure. It says that, for every integrable function \(f:\R^n\to\R\),
		\begin{equation*}
		\int_{\R^n} f(x)\de x=n\abs{K_0}\int_0^\infty r^{n-1}\int_{\partial K_0} f(rx)\de\mu_{K_0}(x)\de r,
		\end{equation*}
		see \cite[Proposition 1]{NR2002}.
		We apply this transformation formula to \(f(x)=\indic_{K_0}(x)\abs{\xtheta}^{2q}\). Then, we get
		\begin{equation}
		\label{eq:coefficient}
		\begin{split}
		\int_{\R^n}\abs{\xtheta}^{2q}\de\nu_{K_0}(x)&=	\int_{\R^n}\abs{\xtheta}^{2q}\frac{\indic_{K_0}(x)}{\abs{K_0}}\de x\\
		&=n\int_0^\infty r^{n-1}\int_{\partial K_0}\abs{\langle rx,\theta\rangle}^{2q}\indic_{K_0}(rx) \de\mu_{K_0}(x)\de r\\
		&=n\int_0^\infty r^{n-1+2q}\int_{\partial K_0}\abs{\langle x,\theta\rangle}^{2q}\indic_{[0,1]}(r) \de\mu_{K_0}(x)\de r\\
		&=n\int_0^1 r^{n-1+2q}\int_{\partial K_0} \abs{\langle x,\theta\rangle}^{2q}\de\mu_{K_0}(x)\de r\\
		&=\frac{n}{n+2q}\int_{\partial {K_0}} \abs{\langle x,\theta\rangle}^{2q}\de\mu_{K_0}(x),
		\end{split}
		\end{equation}
		which completes the argument.
	\end{proof}
	\begin{remark}
		The quantitative dependence of the constant \(c_{n,q}=n/(n+2q)\) on \(n\) and \(q\) is of importance on its own. This will become clear in the proof of \Cref{lem:psi2B1}.
	\end{remark}
	\begin{lemma}
		\label{lem:psi2cB1}
		For every \(c\in(0,\infty)\),
		\begin{equation}
		\psinorm{c\B_1^n}{\dottheta}=c\psinorm{\B_1^n}{\dottheta}.
		\end{equation}
	\end{lemma}
	\begin{proof}
		It is well known that $\mu_{c\B_1^n}$ coincides with the normalized $(n-1)$ Hausdorff measure $\mathcal{H}^{n-1}$ on $c\B_1^n$, see \cite{RR91}. Then, for \(t>0\) large enough,
		\begin{align*}
		\int_{c\,\mathbb{S}_1^{n-1}}\exp\bigl((\xtheta/t)^2\bigr)\de\mu_{c\B_1^n}(x)&=\frac{1}{\mathcal{H}^{n-1}(c\B_1^n)}	\int_{c\,\mathbb{S}_1^{n-1}}\exp\bigl((\xtheta/t)^2\bigr)\de\mathcal{H}^{n-1}(x)\\
		&=\frac{c^{n-1}}{\mathcal{H}^{n-1}(c\B_1^n)}	\int_{ \mathbb{S}_1^{n-1}}\exp\bigl((c\langle x',\theta\rangle/t)^2\bigr)\de\mathcal{H}^{n-1}(x')\\
		&=\int_{ \mathbb{S}_1^{n-1}}\exp\bigl((c\langle x',\theta\rangle/t)^2\bigr)\de\mu_{\B_1^n}(x')\,,
		\end{align*}
		where we used that $\mathcal{H}^{n-1}$ is homogeneous of degree $n-1$ in the last step.
		This implies the claim by definition of the $\psi_2$-norm.
	\end{proof}
	\begin{lemma}
		\label{lem:psi2B1}
		For every \(\theta\in\R^n\),
		\begin{equation*}
		\sqrt{n}\psinorm{\B_1^n}{\dottheta}\le \sqrt{6}\maxnorm{\theta}.
		\end{equation*}
	\end{lemma}
	\begin{proof}
		The proof follows the one in \cite{BobkovNazarov} (see also page 305 in \cite{IsotropicConvexBodies}). Let \(q\in\N\cup\{0\}\). Using the unconditionality of \(\B_1^n\), expanding the power of the scalar product yields
		\begin{equation*}		
		\int_{\R^n}\abs{\xtheta}^{2q}\de\nu_{\B_1^n}(x)=\sum_{q_i\in\N\cup\{0\},\sum_{i=1}^n q_i=q}\binom{2q}{2q_1,\ldots,2q_n}\prod_{i=1}^n\theta_i^{2q_i}\int_{\R^n}\prod_{i=1}^nx_i^{2q_i}\de\nu_{\B_1^n}(x),
		\end{equation*}
		where we used the standard notation for multinomial coefficients.
		Moreover, whenever we have \(q_1+\ldots+q_n=q\), it holds
		\begin{equation*}
		\int_{\R^n}\prod_{i=1}^nx_i^{2q_i}\de\nu_{\B_1^n}(x)=\frac{n!}{(n+2q)!}\prod_{i=1}^{n}(2q_i)!\,.
		\end{equation*}
		For the sake of completeness, we prove this claim by induction. Note that it is equivalent to
		\begin{equation*}
		\int_{\B_1^n}\prod_{i=1}^nx_i^{2q_i}\de x=\frac{2^n}{(n+2q)!}\prod_{i=1}^{n}(2q_i)!\,.
		\end{equation*}
		The equality holds for \(n=1\), indeed both sides are equal to \(2/(1+2 q_1)\). Suppose that it holds in dimension \(n\) and for exponents \(2q_1,\ldots,2q_n\) whose sum is equal to \(2q\). We want to prove it in dimension \(n+1\) adding a new exponent \(2q_{n+1}\):
		\begin{align*}
		\int_{\B_1^{n+1}}\prod_{i=1}^{n+1} x_i^{2q_i}\de x_1\ldots\de x_{n+1}&=\int_{-1}^{1}\int_{(1-\abs{x_{n+1}})\B_1^n}\prod_{i=1}^nx_i^{2q_i}\de x_1\ldots\de x_n\,x_{n+1}^{2q_{n+1}}\de x_{n+1}\\
		&=\int_{-1}^1\int_{\B_1^n}(1-\abs{x_{n+1}})^{n+2q}\,\prod_{i=1}^{n}y_i^{2q_i}\de y_1\ldots \de y_n\,x_{n+1}^{2q_{n+1}}\de x_{n+1}\\
		&=\frac{2^n}{(n+2q)!}\prod_{i=1}^{n}(2q_i)!\int_1^1 (1-\abs{x_{n+1}})^{n+2q}x_{n+1}^{2q_{n+1}}\de x_{n+1}\\
		&=\frac{2^{n+1}}{(n+2q)!}\prod_{i=1}^{n}(2q_i)!\int_{0}^{1}(1-z)^{n+2q}z^{2q_{n+1}}\de z\\
		&=\frac{2^{n+1}}{(n+2q)!}\prod_{i=1}^{n}(2q_i)!\frac{(n+2q)!(2q_{n+1})!}{(n+2q+2q_{n+1}+1)!}\\
		&=\frac{2^{n+1}}{(n+1+2\sum_{i=1}^{n+1}q_i)!}\prod_{i=1}^{n+1}(2q_i)!,
		\end{align*}
		which proves the claim.
		
		Now, if we set \(\alpha\coloneqq\sqrt{n}\maxnorm{\theta}\), then it holds that \(\prod_{i=1}^n\theta_i^{2q_i}\le \alpha^{2q}n^{-q}\).
		This yields
		\begin{align*}
		\int_{\R^n}\abs{\xtheta}^{2q}\de\nu_{\B_1^n}(x)&\le \sum_{q_i\in\N\cup\{0\},\sum_{i=1}^n q_i=q}\binom{2q}{2q_1,\ldots,2q_n}\frac{\alpha^{2q}}{n^q}\frac{n!}{(n+2q)!}\prod_{i=1}^{n}(2q_i)!\\
		&=\sum_{q_i\in\N\cup\{0\},\sum_{i=1}^n q_i=q} \frac{n!(2q)!\,\alpha^{2q}}{(n+2q)!\,n^q}\\
		&=\binom{n+q-1}{n-1}\frac{n!(2q)!\,\alpha^{2q}}{(n+2q)!\,n^q},
		\end{align*}
		where in the last equality we used that the cardinality of the set of indices in the sum is precisely $\binom{n+q-1}{n-1}$. Using \Cref{eq:coefficient}, we get
		\begin{align*}
		\int_{\R^n}\abs{\xtheta}^{2q}\de\mu_{\B_1^n}(x)&=\frac{n+2q}{n}	\int_{\R^n}\abs{\xtheta}^{2q}\de\nu_{\B_1^n}(x)\\
		&\le\binom{n+q-1}{n-1}\frac{(n-1)!(2q)!\,\alpha^{2q}}{(n+2q-1)!\,n^q}\\
		&=\oneover{(n+q)\cdots(n+2q-1)}\frac{(2q)!\,\alpha^{2q}}{q!\,n^q}\\
		&\le\frac{q!}{2}\Bigl(\frac{2\alpha}{n}\Bigr)^{2q},
		\end{align*}
		where for the last step we used the inequality \(2(2q)!\le (2^q q!)^2\), which can be checked by induction on \(q\in\N\), and the fact that \(n^q\le(n+q)\cdots(n+2q-1)\).
		When \( \abs{t}<1/(2\alpha)\), we have
		\begin{align*}
		\int_{\R^n}\exp\bigl((tn{\xtheta})^{2}\bigr)\de\mu_{\B_1^n}(x)&=1+\sum_{q=1}^\infty \frac{t^{2q}n^{2q}}{q!}	\int_{\R^n}\abs{\xtheta}^{2q}\de\mu_{\B_1^n}(x)\\
		&\le 1+\frac{1}{2}\sum_{q=1}^\infty(2t\alpha)^{2q}\\
		&=1+\oneover{2}\Bigl(\oneover{1-4t^2\alpha^2}-1\Bigr).
		\end{align*}
		If \(t_0\) is such that the last expression equals \(2\) when evaluated in \(t=t_0\), we get that 
		\begin{equation*}
		n\psinorm{\B_1^n}{\dottheta}\le 1/t_0=\sqrt{6}\alpha=\sqrt{6}\sqrt{n}\maxnorm{\theta},
		\end{equation*}
		which completes the proof.
	\end{proof}
	
	\begin{proof}[Proof of \Cref{prop:psi2}]
	From \Cref{lem:comparison} and the definition of \(\psinorm{V}{\dottheta}\), we get 
	\begin{align*}
	\int_{\partial K}\exp\Bigl(\xtheta^2/&\psinorm{V}{\dottheta}^2\Bigr)\de\mu_K(x)\\
	&=1+\sum_{q=1}^\infty\oneover{q!\psinorm{V}{\dottheta}^{2q}}	\int_{\partial K}\abs{\langle x,\theta\rangle}^{2q}\de\mu_K(x)\\
	&\le 1+\sum_{q=1}^\infty\oneover{q!\psinorm{V}{\dottheta}^{2q}}	\int_{\partial V}\abs{\langle x,\theta\rangle}^{2q}\de\mu_V(x)\\
	&= 2.
	\end{align*}
	In particular, we have
	\begin{equation*}
	\psinorm{K}{\dottheta}\le\psinorm{V}{\dottheta}.
	\end{equation*}
	Moreover, from \Cref{lem:psi2cB1} and \Cref{lem:psi2B1}, we obtain
	\begin{equation*}
	\psinorm{V}{\dottheta}=\frac{\sqrt{6}}{2}n\psinorm{\B_1^n}{\dottheta}\le 3\sqrt{n}\maxnorm{\theta}.
	\end{equation*}
	The proof is thus complete.
	\end{proof}
	
	\begin{remark}\label{rem:ExtremalBodies}
	Let us emphasize that we decided to use a different approach than the one used in \Cref{lem:Psi1EstimateGeneral} below in order to gain a comparison between the cone measure of an isotropic unconditional convex body and the cone measure of a suitably rescaled ball with respect to the $1\text{-norm}$. Also, we have made explicit every constant in our computations.
	\end{remark}
	
	\section{Proof of Theorem \ref{thm:IsotropicConstant}}\label{sec:ProofIsotropicConstant}
	
	Recalling the bound for the isotropic constant presented in \Cref{lem:boundLK} (a), our proof is naturally divided into two parts. The first is concerned with a lower bound on the volume radius of our random polytope. 
	
	In the following Lemma, part \((a)\) will be applied to the case of an isotropic unconditional convex body, while part \((b)\) will be used for the general case of an isotropic convex body.
	\begin{lemma}\label{lem:volrad}
		Let $K\subset\R^n$ be a convex body with \(\vol{K}=1\) and $K_N$ the symmetric convex hull of $N$ independent random points on \(\partial K\) with distribution $\mu_K$.
	\begin{itemize}
	\item[(a)] There exist constants \(c_1\in(1,\infty)\) and \(c_2\in(0,\infty)\) such that the event  that
			\begin{equation*}
			\abs{K_N}^{1/n}\ge c_2\min\biggl\{\sqrt{\frac{\log(2N/n)}{n}},1\biggr\}
			\end{equation*}
			has probability greater than \(1-\exp(-n)\) when \(N\ge c_1 n\).
	\item[(b)]There exist constants \(c_1\in(1,\infty)\) and \(c_2\in(0,\infty)\) such that the event  that
			\begin{equation*}
			\abs{K_N}^{1/n}\ge c_2\,L_K\sqrt{\frac{\log(2N/n)}{n}}
			\end{equation*}
			has probability greater than \(1-\exp(-c_1\sqrt{N})\) when \(n\leq N\leq e^{\sqrt{n}}\). 
	\end{itemize}
	\end{lemma}
	\begin{proof}		
		Let us start with (a). We use a coupling argument that was introduced in \cite{HPT17}. Let \(Y_1,\ldots,Y_N\) be independent random points distributed according to the uniform distribution on \(K\), and define the symmetric random polytope
			\begin{equation*}
			\widetilde{K}_N\coloneqq\mathrm{conv}\{\pm Y_1,\ldots,\pm Y_N\}.
			\end{equation*}
		It is proven in \cite[Proposition 2.2]{DGG} that if \(N\ge c_1 n\), then
			\begin{equation*}
			\vol{\widetilde{K}_N}^{1/n}\ge c_2\min\biggl\{\sqrt{\frac{\log(2N/n)}{n}},1\biggr\}
			\end{equation*}
		with probability greater than \(1-\exp(-n)\). For \(i\in\{1,\ldots,N\}\), consider the random variables
	
			\begin{equation*}
				X_i\coloneqq\begin{dcases}
					\frac{Y_i}{\norm{Y_i}_K}&:\, \norm{Y_i}_K\neq 0\\
					y\in\partial K&:\, \norm{Y_i}_K= 0\,,
				\end{dcases}
			\end{equation*}
		where \(y\) is a fixed but arbitrary point on \(\partial K \). By definition, the points \(X_1,\ldots,X_N\) are independent and belong to \(\partial K \).  Moreover, the push-forward probability measure of the uniform distribution $\nu_K$ under the map $K\ni y\mapsto y/\norm{y}_K\in\partial K$ is exactly the cone probability measure \(\mu_K \) on $K$. Indeed, for any Borel set \(B\subset\partial K\), 
		\begin{equation*}
		\P(X_i\in B)=\P(Y_i\in(0,1]B)=\frac{\abs{(0,1]B}}{\vol{K}}=\mu_K(B).
		\end{equation*}
	Note also that it follows from the symmetry of \(K\) that if \(X\in\partial K\), then also \(-X\in\partial K\). In particular, the symmetric random polytope
		\begin{equation*}
			K_N\coloneqq\mathrm{conv}\{\pm X_1,\ldots, \pm X_N\} 
		\end{equation*}	
	has the desired distribution. Moreover, by construction \(K_N(\omega)\supseteq\widetilde{K}_N(\omega)\) for every realization \(\omega\in\Omega\), so that 
		\begin{equation*}
			\vol{K_N}^{1/n}\ge\vol{\widetilde{K}_N}^{1/n}\ge c_2\min\biggl\{\sqrt{\frac{\log(2N/n)}{n}},1\biggr\}
		\end{equation*}
		with probability greater than \(1-\exp(-n) \).
		
		The proof of part (b) is similar. The only change is that for the lower bound for $\vol{\widetilde{K}_N}^{1/n}$ instead of \cite[Proposition 2.2]{DGG} we now use \cite[Theorem 4,1]{DGT1} in the form of \cite[Theorem 11.3.7]{IsotropicConvexBodies}.
		\end{proof}
	
	Now that we have established the \(\psi_2\text{-estimate}\) in the previous section, we can proceed to bound the second quantity that we need in view of \Cref{lem:bernstein} (a).
	
	\begin{lemma}
		\label{prop:1normbound}
		Let \(K\subset\R^n\) be an isotropic unconditional convex body. For \(N>n\) let \(X_1,\ldots,X_N\) be independent random points distributed according to the cone measure on \(\partial K \). Then there exist constants \(c,C\in(0,\infty) \) such that with probability greater than \(1-\exp(-cn\log(2N/n))\) it holds 
		\begin{equation*}
		\max_{\epsilon_1,\ldots,\epsilon_n=\pm 1}\norm{\epsilon_1X_{i_1}+\ldots+\epsilon_nX_{i_n}}_1\le C\,n^{3/2}\sqrt{\log(2N/n)}
		\end{equation*}
		for all subsets of vertices \(\{X_{i_1},\ldots,X_{i_n}\}\subset\{\pm X_1,\ldots, \pm X_N\}\).
	\end{lemma}
	\begin{proof}
		We start considering the points \(X_1,\ldots,X_n\). Fix a direction \(\theta\in\S_\infty \) and an \(n\text{-tuple}\) of signs \(\epsilon=(\epsilon_1,\ldots,\epsilon_n)\in\{-1,+1\}^n\). For every \(i\in\{1,
		\ldots, n\}\), we define the random variables \(Y_i\coloneqq \langle \epsilon_{i}X_i,\theta\rangle\). Note that by \Cref{prop:psi2}, \(\psinorm{K}{Y_i}\le 3\sqrt{n} \) so that we can apply linearity and the $\psi_2$-version of Bernstein's inequality (see \Cref{lem:bernstein} (a)) in order to get
		\begin{equation}\label{eq:ProofMax1}
		\P\bigl(\abs{\langle \epsilon_1X_{1}+\ldots+\epsilon_nX_{n},\theta \rangle}>t n\bigr)\le 2\exp(-t^2/72),
		\end{equation}
		for every \(t>0\). Now we notice that
		\begin{equation*}
		\norm{\epsilon_1X_{1}+\ldots+\epsilon_nX_{n}}_1=\sup_{\theta\in\S_\infty}\abs{\langle \epsilon_1X_{1}+\ldots+\epsilon_nX_{n},\theta \rangle}=\max_{\theta\in\{-1,1\}^n}\abs{\langle \epsilon_1X_{1}+\ldots+\epsilon_nX_{n},\theta \rangle}.
		\end{equation*}
		Hence, we obtain
		\begin{equation}\label{eq:ProofMax2}
		\begin{split}
		\P\biggl(\,\max_{\epsilon\in\{-1,+1\}^n} 	&\norm{\epsilon_1X_{1}+\ldots+\epsilon_nX_{n}}_1>t n\biggr)\\
		&=\P\biggl(\,\max_{\epsilon,\theta\in\{-1,1\}^n} \abs{\langle \epsilon_1X_{1}+\ldots+\epsilon_nX_{n},\theta \rangle}>t n\biggr)\\
		&\le 4^n \P\bigl(\abs{\langle \epsilon_1X_{1}+\ldots+\epsilon_nX_{n},\theta \rangle}>t n\bigr)\\
		&\le\exp\big((2n+1)\log 2-t^2/72\big),
		\end{split}
		\end{equation}
		where we used the union bound together with the fact that \(\abs{\langle \epsilon_1X_{1}+\ldots+\epsilon_nX_{n},\theta \rangle}\) has the same distribution for every choice of signs \(\epsilon_i\)'s and directions \(\theta\). We now consider all the subsets \(\{X_{i_1},\ldots,X_{i_n}\}\subset \{\pm X_1,\ldots,\pm X_n\}\) of cardinality \(n\). Since there are \(\binom{2N}{n}\le (2eN/n)^n=\exp(n\log(2N/n))\) of such subsets, we can set \(t\coloneqq C\sqrt{n\log(2N/n)}\), with \(C\in(0,\infty)\) sufficiently large, and use again the union bound to get 
		\begin{equation}\label{eq:ProofMax3}
		\begin{split}
		\P\biggl(\max_{\{X_{i_1},\ldots,X_{i_n}\}\subset\{\pm X_1,\ldots,\pm X_N\} }\max_{\epsilon\in\{-1,+1\}^n} 	&\norm{\epsilon_1X_{i_1}+\ldots+\epsilon_nX_{i_n}}_1>C n^{3/2}\sqrt{\log(2N/n)}\biggr)\\
		&\le\exp\big((2n+1)\log 2-(C^2/72-1)n\log(2N/n)\big)\\
		&\le\exp\big(-cn\log(2N/n)\big),
		\end{split}
		\end{equation}
		which implies the statement by taking the complementary event.
	\end{proof}
	
	We are now prepared to complete the proof of \Cref{thm:IsotropicConstant}.
	
	\begin{proof}[Proof of \Cref{thm:IsotropicConstant}]
	By \Cref{rem:N<cn} the conclusion is clear if $N\leq cn$ for some constant $c\in(0,\infty)$.
	
	Let us next assume that there are constants $c_0,c_1\in(0,\infty)$ such that $c_0n\leq N\leq e^{c_1n}$. Since every facet of \(K_N\) is obtained as the convex hull of a subset (of cardinality \(n\) with probability one) of all the vertices, \Cref{prop:1normbound} together with \Cref{lem:facets} (a) immediately gives that
	\begin{equation}\label{eq:ProofThm11}
		\oneover{\abs{K_N}}\int_{K_N}\norm{x}_1\de x\le (1+\sqrt{2})C\sqrt{n \log(2N/n)}
	\end{equation}
with probability greater than \(1-\exp(-c n)\), where $c,C\in(0,\infty)$ are the same constants as in \Cref{prop:1normbound}. Combining this with \Cref{lem:boundLK} (a) and \Cref{lem:volrad} (a), we get that 
		\begin{equation}\label{eq:ProofThm12}
		\begin{split}
			L_{K_N}&\le\frac{c_4}{n}\oneover{\abs{K_N}^{1/n}}	\oneover{\abs{K_N}}\int_{K_N}\norm{x}_1\de x\\
			&	\le c \cdot c_2\cdot \oneover{n}\sqrt{\frac{n}{\log(2N/n)}}\cdot(1+\sqrt{2})\, C\sqrt{n\log(2N/n)}\\
			&=(1+\sqrt{2})\,c\cdot c_2\cdot C
		\end{split}
		\end{equation}
with probability greater than  \(1-c_3\exp(-c_4n)\).

Finally, we treat the case where $N\geq e^{an}$ for some constant $a\in(0,\infty)$. In this case \Cref{lem:volrad} (a) yields that $\vol{K_N}^{1/n}\geq c_2$ for some constant $c_2\in(0,\infty)$ with probability at least $1-e^{-n}$. In addition, by unconditionality of \(K\) it holds that \(K\subset (\sqrt{6}/2)n\mathbb{B}_1^n\), hence
\begin{equation*}
\oneover{\vol{K_N}}\int_{K_N}\norm{x}_1\,\de x \leq \frac{\sqrt{6}/2}{\vol{K_N}}\int_{K_N}n\norm{x}_{K_N}\de x \leq \frac{\sqrt{6}}{2}n.
\end{equation*}
Thus, \Cref{lem:boundLK} (a) yields the bound
\begin{equation*}
L_{K_N} \leq \frac{c}{n}\oneover{\vol{K_N}^{1/n}}\oneover{\vol{K_N}}\int_{K_N}\norm{x}_1\de x \leq \frac{c}{n}\oneover{c_2}\frac{\sqrt{6}}{2}n=\frac{\sqrt{6}}{2}\frac{c}{c_2}
\end{equation*}
with probability at least $1-e^{-n}$. The proof is thus complete.
	\end{proof}

\section{Proof of \Cref{thm:IsotropicConstantGeneralK}}\label{sec:ProofGeneral}

In this section we give a proof of \Cref{thm:IsotropicConstantGeneralK}. We start with the following $\psi_1$-estimate.

\begin{lemma}\label{lem:Psi1EstimateGeneral}
Fix an isotropic convex body $K\subset\R^n$ and $\theta\in\mathbb{S}^{n-1}$. Then there exists an absolute constant $c\in(0,\infty)$ such that $\norm{\langle\cdot,\theta\rangle}_{L^{\psi_1}(\mu_K)}\leq c L_K$.
\end{lemma}
\begin{proof}
We recall that \cite[Lemma 2.4.2]{IsotropicConvexBodies} implies that
\begin{align*}
\norm{\langle\cdot,\theta\rangle}_{L^{\psi_1}(\mu_K)} & \leq c\sup_{p\geq 1}\frac{\norm{\langle\cdot,\theta\rangle}_{L^p(\mu_K)}}{p}
\end{align*}
for some absolute constant $c\in(0,\infty)$. Moreover, from \eqref{eq:coefficient} we deduce that
\begin{align*}
\frac{\norm{\langle\cdot,\theta\rangle}_{L^p(\mu_K)}}{p} =\Bigl(\frac{n+p}{n}\Bigr)^{1/p}\frac{\norm{\langle\cdot,\theta\rangle}_{L^p(\nu_K)}}{p},
\end{align*}
where $\nu_K$ is the uniform distribution on $K$. This implies
\begin{align*}
\norm{\langle\cdot,\theta\rangle}_{L^{\psi_1}(\mu_K)} \leq c\,\sup_{p\geq 1}\Bigl(\frac{n+p}{n}\Bigr)^{1/p}\sup_{p\geq 1}\frac{\norm{\langle\cdot,\theta\rangle}_{L^p(\nu_K)}}{p}\leq C\norm{\langle\cdot,\theta\rangle}_{L^{\psi_1}(\nu_K)}
\end{align*}
for another constant $C\in(0,\infty)$, since the first supremum is bounded by $2$. However, $\norm{\langle\cdot,\theta\rangle}_{L^{\psi_1}(\nu_K)}$ is bounded by a constant multiple of \(L_K\), since every isotropic log-concave measure is known to be a so-called $\psi_1\text{-measure}$ (this is essentially a consequence of Borell's lemma, see \cite[page 81]{IsotropicConvexBodies}).
\end{proof}

In a next step we observe that \Cref{lem:volrad} (b) yields a lower bound for $\vol{K_N}^{1/n}$, which depends on the isotropic constant $L_K$ of $K$ whenever $N\leq e^{\sqrt{n}}$. In addition, \Cref{prop:1normbound} needs an adaptation. Especially, while in the unconditional case we could work with the $1$-norm, here we have to deal with the $2$-norm instead. Eventually, this leads to the appearance of the additional logarithmic factor in our final result. Moreover, we have to make explicit now the dependence on $L_K$, since for a general isotropic convex body we do not know whether or not this quantity is bounded by an absolute constant, as explained in the introduction. In the end this will allow us to bound $L_{K_N}$ independently of $L_K$ if $N\leq e^{\sqrt{n}}$. 

\begin{lemma}\label{prop:1NormBoundGeneral}
		Let \(K\subset\R^n\) be an isotropic convex body. For \(N>n\) let \(X_1,\ldots,X_N\) be independent random points distributed according to the cone measure on \(\partial K \). Then there exist constants \(c,C\in(0,\infty) \) such that with probability greater than \(1-\exp(-cn\log(2N/n))\) it holds 
		\begin{equation*}
		\max_{\epsilon_1,\ldots,\epsilon_n=\pm 1}\norm{\epsilon_1X_{i_1}+\ldots+\epsilon_nX_{i_n}}_2\le C L_K n \log(2N/n)
		\end{equation*}
		for all subsets of vertices \(\{X_{i_1},\ldots,X_{i_n}\}\subset\{\pm X_1,\ldots, \pm X_N\}\).
\end{lemma}
\begin{proof}
The proof follows the one of \Cref{prop:1normbound} and we shall indicate the necessary modifications.

Let $X_1,\ldots,X_n$ be independent random points with distribution $\mu_K$ and, for $\theta\in\mathbb{S}^{n-1}$ and $\varepsilon_1,\ldots,\varepsilon_n\in\{-1,+1\}$, put $Y_i\coloneqq\langle\varepsilon_iX_i,\theta\rangle$ for any $i\in\{1,\ldots,n\}$. We start by noticing that Lemma \ref{lem:Psi1EstimateGeneral} implies that if $K\subset\R^n$ is an arbitrary isotropic convex body, we have that $\norm{Y_i}_{L^{\psi_1}(\mu_K)}\leq c L_K$ for some absolute constant $c\in(0,\infty)$ and any $i\in\{1,\ldots,n\}$. 

Thus, we can apply the $\psi_1$-version of Bernstein's inequality (\Cref{lem:bernstein} (b)), which implies that \eqref{eq:ProofMax1} needs to be replaced by
\begin{equation*}
\P\bigl(\abs{\langle \epsilon_1X_{1}+\ldots+\epsilon_nX_{n},\theta \rangle}> p c L_K n\bigr)\le 2\exp(-p n/6),
\end{equation*}
for some parameter \(p>1\) to be chosen later.
Taking the union bound, we get
\begin{equation*}
\P\Bigl(\,\max_{\epsilon\in\{-1,+1\}^n} 	\abs{\langle \epsilon_1X_{1}+\ldots+\epsilon_nX_{n},\theta \rangle}>pc L_K n\Bigr)\le\exp\big((n+1)\log 2-pn/6\big).
\end{equation*}
Consider now a \(\tfrac12\text{-net}\) \(\mathcal{N}\) of \(\mathbb{S}^{n-1}\) with cardinality at most $5^n$ (the existence of such a net is ensured by \cite[Lemma 5.2.5]{AsymGeomAnalyBook}, for example). Applying the union bound once more leads to
\begin{equation*}
\P\Bigl(\max_{\theta\in\mathcal{N}}\max_{\epsilon\in\{-1,+1\}^n} 	\abs{\langle \epsilon_1X_{1}+\ldots+\epsilon_nX_{n},\theta \rangle}>pcL_K n\Bigr)\le\exp\big((n+1)\log 2+n\log 5 -pn/6\big).
\end{equation*}
For any \(\theta\in\mathbb{S}^{n-1}\) there exist a sequence \((\theta_j)_{j\in\N}\in\mathcal{N}^\N \) and coefficients \(\delta_j\in[0,2^{1-j}]\) such that \(\theta=\sum_{j=1}^\infty \delta_j\theta_j \)  (see \cite{A}). In particular, this implies
\begin{equation*}
\begin{split}
\P\Bigl(\max_{\theta\in\mathbb{S}^{n-1}}&\max_{\epsilon\in\{-1,+1\}^n} 	\abs{\langle \epsilon_1X_{1}+\ldots+\epsilon_nX_{n},\theta \rangle}>2pc L_K n\Bigr)\\
&\le \P\Bigl(\max_{\theta\in\mathbb{S}^{n-1}}\max_{\epsilon\in\{-1,+1\}^n} \sum_{j=1}^\infty \delta_j\abs{\langle\epsilon_1X_{1}+\ldots+\epsilon_nX_{n},\theta_j \rangle}>2pc L_K n\Bigr)\\
&\le   \P\Bigl(\max_{\theta\in\mathcal{N}}\max_{\epsilon\in\{-1,+1\}^n}\abs{\langle\epsilon_1X_{1}+\ldots+\epsilon_nX_{n},\theta_j \rangle}>pc L_K n\Bigr)\\
& \le\exp\big((n+1)\log 2+n\log 5 -pn/6\big).
\end{split}
\end{equation*}
Notice that 
\begin{equation*}
	\max_{\theta\in\mathbb{S}^{n-1}}\abs{\langle \epsilon_1X_{1}+\ldots+\epsilon_nX_{n},\theta \rangle}=\norm{ \epsilon_1X_{1}+\ldots+\epsilon_nX_{n}}_2.
\end{equation*}
Hence, applying a union bound and taking $p\coloneqq42\log(2N/n)$, \eqref{eq:ProofMax3} gets replaced by
		\begin{align*}
		\P\biggl(\max_{\{X_{i_1},\ldots,X_{i_n}\}\subset\{\pm X_1,\ldots,\pm X_N\} }\max_{\epsilon\in\{-1,+1\}^n} 	&\norm{\epsilon_1X_{i_1}+\ldots+\epsilon_nX_{i_n}}_2>84c L_K n\,\log(2N/n)\biggr)\\
		&\le\exp\big(-n\log(2N/n)\big).
		\end{align*}
This completes the proof.
\end{proof}

\begin{proof}[Proof of \Cref{thm:IsotropicConstantGeneralK}]
Again, the proof follows closely the one of Theorem \ref{thm:IsotropicConstant} and we shall indicate the necessary modifications.

The regime where $N\leq cn$ is trivial by \Cref{rem:N<cn}. Next, as long as $N\leq e^{\sqrt{n}}$ we combine this time \Cref{prop:1NormBoundGeneral} with \Cref{lem:facets} (b) to see that \eqref{eq:ProofThm11} gets replaced by
\begin{equation*}
\oneover{\abs{K_N}}\int_{K_N}\norm{x}_2^2\de x\le 2 C^2 L_K^2 \log(2N/n)^2,
\end{equation*}
which holds with probability greater than \(1-\exp(-c_1n)\), where $C\in(0,\infty)$ is an absolute constant. Combining this with \Cref{lem:boundLK} (b) and \Cref{lem:volrad} (b), we deduce that \eqref{eq:ProofThm12} has to be replaced by
\begin{align*}
			L^2_{K_N}&\le\oneover{n\abs{K_N}^{2/n}}	\oneover{\abs{K_N}}\int_{K_N}\norm{x}_2^2\de x \leq \frac{1} {n}\frac{n}{c_2^2 \log(2N/n) L_K^2}C^2 L_K^2\log(2N/n)^2 \leq\frac{C^2}{c_2^2}\log(2N/n),
\end{align*}
which holds with probability greater than  \(1-c_3\exp(-c_4n)-\exp(-c_5\sqrt{N})\). %Finally for $N$ with $N\geq e^{cn}$ we argue exactly as in the proof of \Cref{thm:IsotropicConstant}.
The proof is thus complete. 
\end{proof}

\subsection*{Acknowledgement}

The authors would like to thank Beatrice-Helen Vritsiou (Edmonton), Apostolos Giannopoulos (Athens) and David Alonso-Guti\'errez (Zaragoza) for a number of stimulating discussions on the topic of this paper.

JP has been supported by a \textit{Visiting International Professor Fellowship} from the Ruhr University Bochum and NT by the German Research Foundation (DFG) via Research Training Group RTG 2131 \textit{High dimensional Phenomena in Probability -- Fluctuations and Discontinuity}.
	
\printbibliography

\vspace{1cm}

\footnotesize

\textsc{Joscha Prochno:} School of Mathematics \& Physical Sciences, University of Hull, United Kingdom\\
%Institute of Mathematics and Scientific Computing, University of Graz, Heinrichstra\ss e 36, 8010 Graz, Austria\\
\textit{E-mail}: \texttt{j.prochno@hull.ac.uk}

\bigskip

\textsc{Christoph Th\"ale:} Faculty of Mathematics, Ruhr University Bochum, Germany\\
\textit{E-mail}: \texttt{christoph.thaele@rub.de}

\bigskip

\textsc{Nicola Turchi:}  Faculty of Mathematics, Ruhr University Bochum, Germany\\
\textit{E-mail}: \texttt{nicola.turchi@rub.de}
\end{document}